\documentclass{article}
\usepackage{amsfonts}
\usepackage{amssymb,amsmath}

\newtheorem{theorem}{Theorem}

\newtheorem{lemma}[theorem]{Lemma}

\newenvironment{proof}[1][Proof]{\noindent\textbf{#1.} }{\ \rule{0.5em}{0.5em}}

\begin{document}

\title{Inverse Problem for a Class of Dirac Operators with Spectral
Parameter Contained in Boundary Conditions}
\author{Khanlar R. Mamedov and Ozge Akcay$\thanks{corresponding author}$ \\
\\
Mathematics Department, Mersin University, 33343, Mersin, Turkey \\
\\
e-mail: hanlar@mersin.edu.tr and ozge.akcy@gmail.com}
\date{}
\maketitle

\begin{abstract}
This paper is related to an inverse problem for a class of Dirac operators
with discontinuous coefficient and eigenvalue parameter contained in
boundary conditions. The asymptotic formula of eigenvalues of this problem
is examined. The theorem on completeness of eigenfunctions is proved. The
expansion formula with respect to eigenfunctions is obtained and Parseval
equality is given. Weyl solution and Weyl function are constructed.
Uniqueness theorem of the inverse problem respect to the Weyl function is
proved. \newline
\end{abstract}

\section{Introduction}

\quad Let 
\[
\sigma_{1}=\left(\begin{array}{cc}
0 & i \\
-i & 0
\end{array}\right), \ \ \ \ \ 
\sigma_{2}=\left(\begin{array}{cc}
1 & 0 \\
0 & -1
\end{array}\right), \ \ \ \ \ 
\sigma_{3}=\left(\begin{array}{cc}
0 & 1 \\
1 & 0
\end{array}\right)
\] be the well-known Pauli-matrices, which has these properties: $\sigma^{2}_{i}=I,$ ($I$ is $2\times2$ identity matrix) 
$\sigma_{i}^{*}=\sigma_{i}$ (self-adjointness)
$i=1,2,3$ and for $i\neq j$, $\sigma_{i}\sigma_{j}=-\sigma_{j}\sigma_{i}$ (anticommutativity).

We consider the following boundary value problem generated by the canonical Dirac system
\begin{equation} \label{1}
By^{\prime}+\Omega \left( x\right) y=\lambda \rho \left( x\right) y,\ \ \ \ \ 0<x<\pi 
\end{equation}
with boundary conditions
\begin{equation} \label{2}
\begin{array}{cc}
U_{1}(y):= b_{1}y_{2}\left( 0\right) +b_{2}y_{1}\left( 0\right)
-\lambda \left( b_{3}y_{2}\left( 0\right) +b_{4}y_{1}\left( 0\right)\right) =0, \\
\\
U_{2}(y):= c_{1}y_{2}\left( \pi \right) +c_{2}y_{1}\left( \pi \right)
 +\lambda \left( c_{3}y_{2}\left( \pi \right) +c_{4}y_{1}\left( \pi\right) \right) =0,
\end{array}
\end{equation}
where
\[
B=\frac{1}{i}\sigma_{1}, \ \ \ \ \Omega(x)=\sigma_{2}p(x)+\sigma_{3}q(x) , \ \ \ \ y\left( x\right)=\left( 
\begin{array}{c}
y_{1}\left( x\right) \\ 
y_{2}\left( x\right)
\end{array}
\right),
\] $p(x),$ $q(x)$ are real measurable functions, $p(x)\in L_{2}(0,\pi ),$ $q(x)\in L_{2}(0,\pi),$ $\lambda$ is a spectral parameter, 
\[
\rho \left( x\right) =\left\{ 
\begin{array}{c}
1,\ \ \ \ \ 0\leq x\leq a, \\ 
\alpha ,\ \ \ \ \ a<x\leq \pi ,
\end{array}
\right.
\] and $1\neq \alpha >0.$ Let us define $k_{1}=b_{1}b_{4}-b_{2}b_{3}>0,$ $k_{2}=c_{1}c_{4}-c_{2}c_{3}>0.$
The main aim of this paper is to solve the inverse problem for the boundary value problem (\ref{1}), (\ref{2})
by Weyl function on a finite interval.

The inverse problem and the spectral properties of Dirac operators were investigated in detail by many authors: 
\cite{Ab, Ag, Cl, Col1, Sa3, Fr, GD, DL, Ga, Ho, K, Kur1, Kur2, La, Lev, Li, Mam, Col2, Mo, Na, Sa, Th, Wa, Ya1, Ya2, Ya3, Yur}. 
The inverse spectral problems according to two spectra was solved in 
\cite{GD}. 
Using Weyl-Titchmarsh function, direct and inverse problems for Dirac type-system were studied in \cite{Sa3, Fr, Sa}.
Solution of the inverse quasiperiodic problem for Dirac system was given in \cite{Na}. 
For weighted Dirac system, inverse spectral problems was examined in \cite{Wa}.
Reconstruction of Dirac operator from nodal data was carried out in \cite{Ya2}.
Necessary and sufficient conditions for the solution of Dirac operators with discontinuous 
coefficient was obtained in \cite{Mam}. 
Inverse problem for interior spectral data of the Dirac operator was given in \cite{Mo}. 
For Dirac operator, Ambarzumian-type theorems were proved in \cite{K, Ya3}. 
On a positive half
line, inverse scattering problem for a system of Dirac equations of order $2n$
was completely solved in \cite{Ga} and when boundary condition contained spectral
parameter, for Dirac operator, inverse scattering problem was worked in \cite{Col1, Col2}.
Spectral boundary value problem in a $3$ dimensional bounded domain for the Dirac system was studied in \cite{Ag}.  
The applications of Dirac differential equations system has been widespread
in various areas of physics, such as \cite{Ok1, Ok2, Mon, Th}.
  
This paper is organized as follows: in section 2, the operator formulation of problem (\ref{1}),(\ref{2}) and some spectral properties of the operator are given. In section 3, asymptotic formula of eigenvalues of the problem (\ref{1}),(\ref{2}) is examined. In section 4, we prove completeness theorem of eigenfunctions and obtain Parseval equality. In section 5, Weyl solution, Weyl function are defined and uniqueness theorem for inverse problem according to Weyl function is proved. 
 
\section{Operator Formulation and Some Spectral Properties}
An inner product in Hilbert space $H_{\rho }=L_{2,\rho }(0,\pi;\mathbb{C}^{2})\oplus \mathbb{C}^{2}$ is given by
\begin{equation} \label{3}
\left\langle Y,Z\right\rangle =\int_{0}^{\pi }\left\{ y_{1}\left( x\right) 
\overline{z_{1}\left( x\right) }+y_{2}\left( x\right) \overline{z_{2}\left(
x\right) }\right\} \rho \left( x\right) dx+\frac{1}{k_{1}}y_{3}\overline{
z_{3}}+\frac{1}{k_{2}}y_{4}\overline{z_{4}},  
\end{equation}
where
\[
Y=\left( 
\begin{array}{c}
y_{1}(x) \\ 
y_{2}(x) \\ 
y_{3} \\ 
y_{4}
\end{array}
\right) \in H_{\rho }, \ \ \ \ \ Z=\left( 
\begin{array}{c}
z_{1}(x) \\ 
z_{2}(x) \\ 
z_{3} \\ 
z_{4}
\end{array}
\right) \in H_{\rho }.
\] Let us define the operator $L$: 
\[
L(Y):=\left( 
\begin{array}{c}
l(y) \\ 
b_{1}y_{2}\left( 0\right) +b_{2}y_{1}(0) \\ 
-\left( c_{1}y_{2}\left( \pi \right) +c_{2}y_{1}(\pi )\right)
\end{array}
\right)
\] with domain 
\[
D(L):=\left\{ Y\mid Y=(y_{1}(x),y_{2}(x),y_{3},y_{4})^{T}\in H_{\rho },\
y_{1}(x),y_{2}(x)\in AC[0,\pi ], \ \ \ \ \ \ \ \ \ \ \ \ \ \ \  \right.
\]
\[\left. \ \ \ \ \ \ \ \ \ \ \ \ \ \ \
y_{3}=b_{3}y_{2}\left( 0\right)+b_{4}y_{1}(0), \ 
y_{4}=c_{3}y_{2}\left( \pi \right) +c_{2}y_{1}(\pi ),
l(y)\in L_{2,\rho}(0,\pi;\mathbb{C}^{2})\right\}
\]
where 
\[
l(y)=\frac{1}{\rho (x)}\left\{ By^{\prime }+\Omega (x)y\right\} .
\] Consequently, the boundary value problem (\ref{1}),(\ref{2}) is equivalent to the
operator equation $LY=\lambda Y$.
\begin{lemma} 
(i) \ The eigenvector functions corresponding to different eigenvalues are
orthogonal.

(ii) \ The eigenvalues of the operator $L$ are real valued.
\end{lemma}

Let $\varphi \left( x,\lambda \right)=\left(\begin{array}{c}
\varphi_{1}(x,\lambda) \\ 
\varphi_{2}(x,\lambda)
\end{array}\right)$ and 
$\psi \left( x,\lambda \right)=\left(
\begin{array}{c}
\psi_{1}(x,\lambda) \\ 
\psi_{2}(x,\lambda)
\end{array}\right)$
be solutions of the system (\ref{1}) satisfying
the initial conditions
\[
\varphi \left( 0,\lambda \right) =\left( 
\begin{array}{c}
\lambda b_{3}-b_{1} \\ 
b_{2}-\lambda b_{4}
\end{array}
\right) ,\ \ \ \ \ \psi \left( \pi ,\lambda \right) =\left( 
\begin{array}{c}
-c_{1}-\lambda c_{3} \\ 
c_{2}+\lambda c_{4}
\end{array}
\right) .
\]

The characteristic function of the problem (\ref{1}),(\ref{2}) is defined by
\begin{equation} \label{4}
\Delta (\lambda )=W[\varphi (x,\lambda ),\psi (x,\lambda )]=\varphi _{2}(x,\lambda )\psi
_{1}(x,\lambda )-\varphi _{1}(x,\lambda )\psi _{2}(x,\lambda),  
\end{equation}
where $W[\varphi (x,\lambda ),\psi (x,\lambda )]$ is Wronskian of the vector
solutions $\varphi (x,\lambda )$ and $\psi (x,\lambda )$. The Wronskian does not depend on $x$. 
It follows from (\ref{4}) that 
\[
\Delta (\lambda )=b_{2}\psi _{1}\left( 0,\lambda \right) +b_{1}\psi
_{2}\left( 0,\lambda \right) -\lambda \left( b_{4}\psi _{1}\left( 0,\lambda
\right) +b_{3}\psi _{2}\left( 0,\lambda \right) \right) =U_{1}\left( \psi
\right)
\]
or 
\[
\Delta (\lambda )=-c_{1}\varphi _{2}\left( \pi ,\lambda \right)
-c_{2}\varphi _{1}\left( \pi ,\lambda \right) -\lambda \left( c_{3}\varphi
_{2}\left( \pi ,\lambda \right) +c_{4}\varphi _{1}\left( \pi ,\lambda
\right) \right) =-U_{2}\left( \varphi \right) .
\]
\begin{lemma} \label{lma} 
The zeros $\lambda _{n}$ of characteristic function coincide
with the eigenvalues of the boundary value problem (\ref{1}),(\ref{2}). The
function $\varphi(x,\lambda _{n})$ and $\psi (x,\lambda _{n})$ are
eigenfunctions and there exist a sequence $\beta _{n}$ such that 
\[
\psi (x,\lambda _{n})=\beta _{n}\varphi (x,\lambda _{n}),\ \ \ \ \beta_{n}\neq 0.
\]
\end{lemma}
\begin{proof}
This lemma is proved by a similar way in \cite{Fy} (see Theorem 1.1.1).
\end{proof}

Norming constants are defined as follows:
\begin{equation} \label{5} 
\alpha_{n}:=\int_{0}^{\pi}\left\{\varphi _{1}^{2}(x,\lambda _{n})
+\varphi_{2}^{2}( x,\lambda _{n}) \right\}
\rho(x) dx
+\frac{1}{k_{1}}\left[b_{3}\varphi _{2}(0,\lambda_{n})+b_{4}\varphi_{1}(0,\lambda _{n}) \right]^{2} 
+\frac{1}{k_{2}}\left[c_{3}\varphi_{2}(\pi ,\lambda _{n}) 
+c_{4}\varphi _{1}(\pi ,\lambda_{n}) \right] ^{2}. 
\end{equation}
\begin{lemma} 
The following relation is valid: 
\[
\alpha _{n}\beta _{n}=\dot{\Delta}(\lambda _{n}),
\] where $\dot{\Delta}(\lambda )=\frac{d}{d\lambda }\Delta (\lambda ).$
\end{lemma}
\begin{proof}
Since $\varphi (x,\lambda )$ and $\psi (x,\lambda )$ are solutions of this
problem, we have 
\begin{eqnarray} \nonumber
\psi _{2}^{^{\prime }}(x,\lambda )+p(x)\psi _{1}(x,\lambda )+q(x)\psi
_{2}(x,\lambda ) &=&\lambda \rho (x)\psi _{1}(x,\lambda ), \\ \nonumber
-\psi _{1}^{^{\prime }}(x,\lambda )+q(x)\psi _{1}(x,\lambda )-p(x)\psi
_{2}(x,\lambda ) &=&\lambda \rho (x)\psi _{2}(x,\lambda ), \\ \nonumber
\varphi _{2}^{^{\prime }}(x,\lambda _{n})+p(x)\varphi _{1}(x,\lambda
_{n})+q(x)\varphi _{2}(x,\lambda _{n}) &=&\lambda _{n}\rho (x)\varphi
_{1}(x,\lambda _{n}), \\ \nonumber
-\varphi _{1}^{^{\prime }}(x,\lambda _{n})+q(x)\varphi _{1}(x,\lambda
_{n})-p(x)\varphi _{2}(x,\lambda _{n}) &=&\lambda _{n}\rho (x)\varphi
_{2}(x,\lambda _{n}). \nonumber
\end{eqnarray}
Multiplying the equations by $\varphi _{1}^{^{\prime }}(x,\lambda_{n}),
\varphi _{2}^{^{\prime}}(x,\lambda _{n}),-\psi _{1}^{^{\prime}}
(x,\lambda ),-\psi _{2}^{^{\prime }}(x,\lambda )$ respectively
and adding them together, we get
\[
\frac{d}{dx}\left\{ \varphi _{1}(x,\lambda _{n})\psi _{2}(x,\lambda )-\psi
_{1}(x,\lambda )\varphi _{2}(x,\lambda _{n})\right\}
=\left( \lambda -\lambda _{n}\right) \rho (x)\left\{ \varphi _{1}(x,\lambda
_{n})\psi _{1}(x,\lambda )+\varphi _{2}(x,\lambda _{n})\psi _{2}(x,\lambda)\right\} .
\]
Integrating it from $0$ to $\pi $,
\[
\left( \lambda -\lambda _{n}\right) \int_{0}^{\pi }\left\{ \varphi
_{1}(x,\lambda _{n})\psi _{1}(x,\lambda )+\varphi _{2}(x,\lambda _{n})\psi
_{2}(x,\lambda )\right\} \rho (x)dx=
\]
\[
=\varphi _{1}(\pi ,\lambda_{n})\psi_{2}(\pi ,\lambda)-
\varphi_{2}(\pi ,\lambda_{n})\psi_{1}(\pi ,\lambda)
-\varphi_{1}(0,\lambda_{n})\psi_{2}( 0,\lambda)+\varphi_{2}(0,\lambda_{n})\psi_{1}
( 0,\lambda)
\]is found. Now, we add 
\[
\left(\lambda-\lambda_{n}\right)\left\{\frac{1}{k_{1}}\left[ b_{3}\psi
_{2}\left( 0,\lambda \right) +b_{4}\psi _{1}\left( 0,\lambda \right)\right]
\left[ b_{3}\varphi _{2}\left( 0,\lambda_{n} \right) +b_{4}\varphi _{1}\left(
0,\lambda_{n} \right) \right]+\right.
\]
\[
\left. +\frac{1}{k_{2}}\left[ c_{3}\psi _{2}\left( \pi ,\lambda \right) +c_{4}\psi
_{1}\left( \pi ,\lambda \right) \right] \left[ c_{3}\varphi _{2}\left( \pi
,\lambda_{n} \right) +c_{4}\varphi _{1}\left( \pi ,\lambda_{n} \right) \right]\right\}
\] in the both sides of last equation and use the boundary condition (\ref{2}).
It follows that
\[
\int_{0}^{\pi }\left\{ \varphi _{1}(x,\lambda _{n})\psi _{1}(x,\lambda
)+\varphi _{2}(x,\lambda _{n})\psi _{2}(x,\lambda )\right\} \rho (x)dx+
\]
\[
+\frac{1}{k_{1}}\left[ b_{3}\psi _{2}\left( 0,\lambda \right) +b_{4}\psi _{1}\left(
0,\lambda \right) \right] \left[ b_{3}\varphi _{2}\left( 0,\lambda_{n} \right)
+b_{4}\varphi _{1}\left( 0,\lambda_{n} \right) \right]+
\]
\[
+\frac{1}{k_{2}}\left[ c_{3}\psi _{2}\left( \pi ,\lambda \right) +c_{4}\psi _{1}\left(
\pi ,\lambda \right) \right] \left[ c_{3}\varphi _{2}\left( \pi ,\lambda_{n}
\right) +c_{4}\varphi _{1}\left( \pi ,\lambda_{n} \right) \right]=\frac{%
\Delta \left( \lambda \right) -\Delta \left( \lambda _{n}\right) }{\lambda -\lambda _{n}}.
\] According to Lemma \ref{lma}, since $\psi (x,\lambda _{n})=\beta _{n}\varphi
(x,\lambda _{n})$, as $\lambda \rightarrow \lambda _{n}$, we obtain
\[
\beta _{n}\alpha _{n}=\dot{\Delta}(\lambda _{n}).
\]
\end{proof}

\section{Asymptotic Formula of Eigenvalues}
\begin{lemma} \label{lma2} 
The solution $\varphi (x,\lambda)=\left(
\begin{array}{c}
\varphi_{1}(x,\lambda) \\ 
\varphi_{2}(x,\lambda)
\end{array}
\right)$ has the following integral representation
\[
\varphi _{1}\left(x,\lambda \right)=\left( \lambda b_{3}-b_{1}\right) \cos
\lambda \mu \left( x\right) +\left( \lambda b_{4}-b_{2}\right) \sin \lambda
\mu \left( x\right)+ \ \ \ \ 
\]
\[
 \ \ \ \ \ \ \ \ \ \ \ +\left( \lambda b_{3}-b_{1}\right) \int_{0}^{\mu
\left( x\right) }\left[\tilde{A}_{11}\left( x,t\right) \cos
\lambda t+\stackrel{\approx }{A}_{12}\left( x,t\right) \sin \lambda t\right]
dt+
\]
\begin{equation}  \label{6}
 \ \ \ \ \ \ \ \ \ \ \ +\left(\lambda b_{4}-b_{2}\right) \int_{0}^{\mu \left( x\right) } 
\left[ \stackrel{\approx}{A}_{11}\left( x,t\right) \sin \lambda t-
\tilde{A}_{12}\left(x,t\right) \cos \lambda t\right] dt,  
\end{equation}

\[
\varphi _{2}\left( x,\lambda \right) =\left(\lambda b_{3}-b_{1}\right) \sin
\lambda \mu \left( x\right) +\left( b_{2}-\lambda b_{4}\right) \cos \lambda
\mu \left( x\right)+ \ \ \ \ 
\]
\[
 \ \ \ \ \ \ \ \ \ \ \ +\left( \lambda b_{3}-b_{1}\right) \int_{0}^{\mu
\left( x\right) }\left[\tilde{A}_{21}\left( x,t\right) \cos
\lambda t+\stackrel{\approx}{A}_{22}\left( x,t\right) \sin \lambda t\right]
dt+
\]
\begin{equation} \label{7}
 \ \ \ \ \ \ \ \ \ \ \ +\left( \lambda b_{4}-b_{2}\right) \int_{0}^{\mu \left( x\right)} 
\left[ \stackrel{\approx}{A}_{21}\left( x,t\right) \sin \lambda t-
\tilde{A}_{22}\left( x,t\right) \cos \lambda t\right] dt,  
\end{equation}
where 
\[
\begin{array}{c}
\tilde{A}_{1j}\left( x,t\right) =K_{1j}\left( x,-t\right)
+K_{1j}(x,t), \ \ \ \ \\ 
\stackrel{\approx }{A}_{1j}\left( x,t\right) =K_{1j}\left( x,t\right)
-K_{1j}(x,-t), \ \ \ \  \\ 
\tilde{A}_{2j}\left( x,t\right) =K_{2j}\left( x,-t\right)
+K_{2j}(x,t), \ \ \ \  \\ 
\stackrel{\approx }{A}_{2j}\left( x,t\right) =K_{2j}\left( x,t\right)
-K_{2j}(x,-t), \ \ \ \ 
\end{array}
\] and $\tilde{A}_{1j}\left( x,.\right) \in L_{2}\left( 0,\pi \right),$ $\stackrel{\approx }{A}_{1j}\left( x,.\right) \in L_{2}\left( 0,\pi
\right) ,$ $\tilde{A}_{2j}\left( x,.\right) \in L_{2}\left( 0,\pi
\right) ,$ $\stackrel{\approx }{A}_{2j}\left( x,.\right) \in L_{2}\left(
0,\pi \right) ,$ $j=1,2.$
\end{lemma}
\begin{proof}
To obtain the form of $\varphi (x,\lambda ),$ we use the integral representation for the solution of
equation (\ref{1}) (detail in \cite{La}). This representation is not operator
transformation and as follows: Assume that
\[
\int_{0}^{\pi}\left\|\Omega(x)\right\|dx<+\infty
\] is satisfied for Euclidean norm of matrix function $\Omega(x).$ Then the
integral representation of the solution of equation (\ref{1}) satisfying the
initial condition $Y(0)=I,$ ($I$ is unite matrix) can be represented
\[
E(x,\lambda )=e^{-\lambda B\mu (x)}+\int_{-\mu \left( x\right)
}^{\mu \left( x\right) }K(x,t)e^{-\lambda Bt}dt,
\] 
where 
\[
\mu (x)=\left\{ 
\begin{array}{c}
x,\ \ \ \ \ \ \ \ \ \ \ \ \ \ \ \ \ \ \ \ \ 0\leq x\leq a, \\ 
\alpha x-\alpha a+a,\ \ \ \ \ \ a<x\leq \pi ,
\end{array}
\right.
\] and for a kernel $K(x,t)$ the inequality 
\[
\int_{-\mu \left( x\right) }^{\mu \left( x\right) }\left\Vert
K(x,t)\right\Vert dt\leq e^{\sigma (x)}-1,
\]
\[
\sigma (x)=\int_{0}^{x}\left\Vert \Omega (s)\right\Vert ds
\] holds. Moreover, if $\Omega(x)$ is differentiable, then $K(x,t)$ satisfy
the following relations 
\[
BK_{x}+\Omega (x)K+\rho (x)K_{t}B=0,
\]
\[
\rho (x)\left[ BK(x,\mu (x))\right] =-\Omega (x),
\]
\[
BK(x,-\mu (x))=0.
\] Now, to find $\varphi(x,\lambda )$, we will use $\varphi(x,\lambda )=E(x,\lambda )\left( 
\begin{array}{c}
\lambda b_{3}-b_{1} \\ 
b_{2}-\lambda b_{4}
\end{array}
\right).$ From the expression of $E(x,\lambda)$
\begin{equation} \label{8}
 \varphi(x,\lambda)=e^{-\lambda B\mu \left( x\right)}\left( 
\begin{array}{c}
\lambda b_{3}-b_{1} \\ 
b_{2}-\lambda b_{4}
\end{array}
\right)+\int_{-\mu(x)}^{\mu(x)}K(x,t)e^{-\lambda Bt}\left( 
\begin{array}{c}
\lambda b_{3}-b_{1} \\ 
b_{2}-\lambda b_{4}
\end{array}
\right) dt  
\end{equation}
can be written. Then
\[
e^{-\lambda B\mu \left( x\right) }\left( 
\begin{array}{c}
\lambda b_{3}-b_{1} \\ 
b_{2}-\lambda b_{4}
\end{array}
\right)=\left( I-\lambda B\mu \left( x\right) +\frac{\left(
-\lambda B\mu \left( x\right) \right) ^{2}}{2!}+...\right) \left( 
\begin{array}{c}
\lambda b_{3}-b_{1} \\ 
b_{2}-\lambda b_{4}
\end{array}
\right)
\]
\[
=\left( 
\begin{array}{c}
\lambda b_{3}-b_{1} \\ 
b_{2}-\lambda b_{4}
\end{array}
\right) -\lambda \left( 
\begin{array}{c}
b_{2}-\lambda b_{4} \\ 
b_{1}-\lambda b_{3}
\end{array}
\right)\mu \left( x\right) +\frac{\lambda ^{2}}{2!}\left( 
\begin{array}{c}
b_{1}-\lambda b_{3} \\ 
\lambda b_{4}-b_{2}
\end{array}
\right) \mu ^{2}\left( x\right)+... \ \ \ \ \ \ \ \ \ \
\]
\[
=\left( 
\begin{array}{c}
\left(\lambda b_{3}-b_{1}\right) \cos \lambda \mu \left( x\right)+
\left(\lambda b_{4}-b_{2}\right) \sin \lambda \mu \left( x\right) \\ 
\left( b_{2}-\lambda b_{4}\right) \cos \lambda \mu \left( x\right) +\left(
\lambda b_{3}-b_{1}\right) \sin \lambda \mu \left( x\right)
\end{array}
\right) .\ \ \ \ \ \ \ \ \ \ \ \ \ \ \ \ \ \ \ \ \ \ \ \ \ \ \ \ \ \ \ \
\] Similar to
\[e^{-\lambda Bt}\left( 
\begin{array}{c}
\lambda b_{3}-b_{1} \\ 
b_{2}-\lambda b_{4}
\end{array}
\right) =\left( 
\begin{array}{c}
\left( \lambda b_{3}-b_{1}\right) \cos \lambda t+\left(\lambda
b_{4}-b_{2}\right) \sin \lambda t \\ 
\left( b_{2}-\lambda b_{4}\right) \cos \lambda t+\left(\lambda
b_{3}-b_{1}\right) \sin \lambda t
\end{array}
\right) .\] Putting these equalities into (\ref{8}), we obtain (\ref{6}) and (\ref{7}).
Moreover, as $\left\vert \lambda \right\vert \rightarrow \infty$ uniformly
in $x\in [ 0,\pi ]$, the following asymptotic formulas hold: 
\begin{equation}  \label{9}
\varphi _{1}(x,\lambda )=\lambda \left( b_{3}\cos \lambda \mu \left(
x\right) +b_{4}\sin \lambda \mu \left( x\right) \right) +O\left(
e^{\left\vert Im\lambda \right\vert \mu \left( x\right) }\right) ,
\end{equation}
\begin{equation} \label{10}
\varphi _{2}(x,\lambda )=\lambda \left( b_{3}\sin \lambda \mu \left(
x\right) -b_{4}\cos \lambda \mu \left( x\right) \right) +O\left(
e^{\left\vert Im \lambda \right\vert \mu \left( x\right) }\right).
\end{equation}
In fact, integrating by parts the integrals involved in (\ref{6}) and (\ref{7}) and also from 
$\left\vert \sin \lambda \mu \left( x\right) \right\vert \leq
e^{\left\vert Im\lambda \right\vert \mu \left( x\right)
}$ and $\left\vert \cos \lambda \mu \left( x\right) \right\vert
\leq e^{\left\vert Im\lambda \right\vert \mu \left(
x\right) },$ the asymptotic formulas (\ref{9}) and (\ref{10}) are found.
\end{proof}
\begin{lemma} 
The eigenvalues $\lambda _{n},(n\in\mathbb{Z})$ of the boundary value problem (\ref{1}),(\ref{2}) are in the form 
\[
\lambda _{n}=\tilde{\lambda}_{n}+\epsilon _{n},
\] where 
\[\tilde{\lambda}_{n}=\left[ n+\frac{1}{\pi }\arctan \left(\frac{
c_{3}b_{4}-c_{4}b_{3}}{b_{3}c_{3}+c_{4}b_{4}}\right)\right] \frac{\pi }{\mu \left(
\pi \right) }\]
and $\{\epsilon_{n}\}\in l_{2}$. Moreover, for the large $n$, the eigenvalues are simple.
\end{lemma}
\begin{proof}
Substituting asymptotic formulas (\ref{9}) and (\ref{10}) into the expression (\ref{4}), we have 
\begin{equation} \label{11}
\Delta \left( \lambda \right) =\lambda^{2}\chi \left( \lambda \right)
+O\left( \left|\lambda\right| e^{\left\vert Im\lambda \right\vert \mu \left(\pi\right) }\right) ,  
\end{equation}
where 
\[\chi \left( \lambda \right) =c_{3}b_{4}\cos \lambda \mu \left( \pi
\right) -b_{3}c_{3}\sin \lambda \mu \left( \pi \right) -c_{4}b_{3}\cos
\lambda \mu \left( \pi \right) -b_{4}c_{4}\sin \lambda \mu \left( \pi \right).
\] Denote 
\[
G_{\delta }:=\left\{ \lambda :\left\vert \lambda -\tilde{\lambda}
_{n}\right\vert \geq \delta ,\ \ n=0,\pm 1,\pm 2...\right\} ,
\] where $\delta $ is a sufficiently small positive number. For $\lambda \in G_{\delta }$, 
\begin{equation}\label{a}
\left\vert \chi \left( \lambda \right) \right\vert \geq C_{\delta }exp\left(
\left\vert Im\lambda \right\vert \mu (\pi )\right)
\end{equation} 
is valid, where $C_{\delta }$ is a positive number. This inequality is similarly obtained as in (\cite{Mar}, Lemma 1.3.2). On the other hand, there
exists a constant $C>0$ such that 
\begin{equation}\label{b}
\left\vert \Delta (\lambda )-\lambda^{2}\chi \left( \lambda \right)
\right\vert \leq C\left\vert \lambda \right\vert e^{\left\vert Im\lambda
\right\vert \mu (\pi)}.
\end{equation} Therefore on infinitely expanding contours 
\[
\Gamma _{n}:=\left\{ \lambda :\left\vert \lambda \right\vert =\tilde{\lambda}_{n}+\frac{\pi }{2\mu (\pi )},\ \ n=0,\pm 1,\pm 2,...\right\} ,
\] for sufficiently large $n$, using (\ref{a}) and (\ref{b}) we get 
\[
\left\vert \Delta (\lambda )-\lambda ^{2}\chi \left( \lambda \right)
\right\vert <\left\vert \lambda \right\vert ^{2}\left\vert \chi \left(
\lambda \right) \right\vert ,\ \ \lambda \in \Gamma _{n}.
\] Applying the Rouche theorem, it is obtained that the number of zeros of the
function $\left\{\Delta (\lambda)-\lambda ^{2}\chi \left( \lambda \right)
\right\} +\lambda ^{2}\chi \left( \lambda \right) =\Delta (\lambda )$ inside
the counter $\Gamma _{n}$ coincides with the number of zeros of function $
\lambda ^{2}\chi \left( \lambda \right).$ Moreover, using the Rouche
theorem, there exist only one zero $\lambda _{n}$ of the function $\Delta
(\lambda )$ in the circle $\gamma _{n}(\delta)=\left\{\lambda :\left\vert
\lambda -\tilde{\lambda}_{n}\right\vert <\delta \right\} $ is
concluded. Since $\delta >0$ is arbitrary, we have 
\begin{equation}  \label{12}
\lambda _{n}=\left[ n+\frac{1}{\pi }\arctan \left(\frac{c_{3}b_{4}-c_{4}b_{3}}{
b_{3}c_{3}+c_{4}b_{4}}\right)\right] \frac{\pi}{\mu \left( \pi \right) }+\epsilon
_{n},\ \lim_{n\rightarrow \pm \infty }\epsilon _{n}=0.  
\end{equation}
Substituting (\ref{12}) into (\ref{11}), we get  
$sin\epsilon _{n}\mu (\pi )=O(\frac{1}{n})$. It follows that $\epsilon _{n}=O(\frac{1}{n})$.
Thus $\epsilon _{n}\in l_{2}$ is found. For the large $n$, the eigenvalues are
simple. In fact, since $\alpha _{n}\beta _{n}=\dot{\Delta}(\lambda _{n})$
and $\alpha _{n}\neq 0,$ $\beta _{n}\neq 0,$ we get $\dot{\Delta}(\lambda
_{n})\neq 0.$ 
\end{proof}

\section{Completeness Theorem}
Firstly, we construct resolvent operator and then we prove completeness theorem. 
The expansion formula respect to eigenfunction is obtained and Parseval equality is given.  
\begin{lemma}
If $\lambda$ is not a spectrum point of operator $L$, then the resolvent
operator exists and has the following form 
\begin{equation}\label{15}
y(x,\lambda )=\int_{0}^{\pi }R_{\lambda }(x,t)f(t)\rho (t)dt+\frac{f_{4}}{
\Delta (\lambda )}\varphi (x,\lambda )+\frac{f_{3}}{\Delta (\lambda )}\psi(x,\lambda ),  
\end{equation}
where 
\begin{equation} \label{16}
R_{\lambda}(x,t)=-\frac{1}{\Delta (\lambda)}\left\{ 
\begin{array}{rl}
\psi (x,\lambda )\widetilde{\varphi }(t,\lambda ),\ \ \  & t\leq x, \\ 
\varphi (x,\lambda )\widetilde{\psi }(t,\lambda ),\ \ \  & t\geq x,
\end{array}
\right.  
\end{equation}
here $\widetilde{\varphi }(t,\lambda )$ denotes the transposed vector
function of $\varphi (t,\lambda ).$
\end{lemma}
\begin{proof}
Let $F(x)=\left( 
\begin{array}{c}
f\left( x\right) \\ 
f_{3} \\ 
f_{4}
\end{array}
\right) \in D(L),$ $f(x)=\left( 
\begin{array}{c}
f_{1}(x) \\ 
f_{2}(x)
\end{array}
\right).$ To construct the resolvent operator of $L$, we solve the
following problem
\begin{equation}\label{17}
By^{\prime}+\Omega \left( x\right) y=\lambda \rho \left( x\right) y+\rho
\left( x\right) f\left( x\right)  
\end{equation}
\begin{equation}\label{18}
\begin{array}{cc}
b_{1}y_{2}\left( 0\right) +b_{2}y_{1}\left( 0\right) -\lambda
\left( b_{3}y_{2}\left( 0\right) +b_{4}y_{1}\left( 0\right) \right) =f_{3}, \\
\\
c_{1}y_{2}\left( \pi \right) +c_{2}y_{1}\left( \pi \right) 
+\lambda \left( c_{3}y_{2}\left( \pi \right) +c_{4}y_{1}\left( \pi \right)
\right)=-f_{4}. 
\end{array}
\end{equation}
The solution of this problem which has a form
\[
y(x,\lambda)=c_{1}(x,\lambda)\varphi(x,\lambda)+C_{2}(x,\lambda)\psi(x,\lambda)
\] is found by applying the method of variation of parameters. Hence (\ref{15}) and (\ref{16}) is obtained.  
\end{proof}
\begin{theorem}
The system of the eigenfunctions $\varphi (x,\lambda _{n}),(n\in 
\mathbb{Z})$ of boundary value problem (\ref{1}),(\ref{2}) is complete in $
L_{2,\rho }(0,\pi;\mathbb{C}^{2})\oplus \mathbb{C}^{2}$.
\end{theorem}
\begin{proof}
Using (\ref{15}) and (\ref{16}) and the equality 
$\psi (x,\lambda _{n})=\frac{\dot{\Delta}(\lambda _{n})}{\alpha _{n}}\varphi
(x,\lambda _{n}),$ (in Lemma \ref{lma}), we get
\begin{equation}  \label{19}
\underset{\lambda =\lambda _{n}}{Res}\ y\left( x,\lambda \right)=
-\frac{1}{\alpha _{n}}\varphi (x,\lambda _{n})\left\{
\int_{0}^{\pi }\widetilde{\varphi }(x,\lambda _{n})f\left( t\right)
\rho \left( t\right) dt-\frac{f_{4}}{\beta _{n}}-f_{3}\right\} .  
\end{equation}
Let $F(x)\in L_{2,\rho}(0,\pi;\mathbb{C}^{2})\oplus \mathbb{C}^{2}$ be such that 
\[
\left\langle F(x),\varphi (x,\lambda _{n})\right\rangle=\int_{0}^{\pi }
\widetilde{\varphi }(t,\lambda _{n})f(t)\rho (t)dt+\frac{1}{k_{1}}f_{3}\left[
b_{3}\varphi _{2}\left( 0,\lambda _{n}\right) +b_{4}\varphi _{1}\left(
0,\lambda _{n}\right) \right]+ 
\]
\[+\frac{1}{k_{2}}f_{4}\left[ c_{3}\varphi _{2}\left( \pi ,\lambda
_{n}\right) +c_{4}\varphi _{1}\left( \pi ,\lambda _{n}\right) \right] =0. 
\]
It follows from the boundary conditions (\ref{2}) and Lemma \ref{lma}, 
\[b_{3}\varphi _{2}\left( 0,\lambda _{n}\right) +b_{4}\varphi _{1}\left(
0,\lambda _{n}\right) =-k_{1}\] and
\[c_{3}\varphi _{2}\left( \pi ,\lambda
_{n}\right) +c_{4}\varphi _{1}\left( \pi ,\lambda _{n}\right) =-\frac{k_{2}}{\beta _{n}}.\] Thus, 
\[
\left\langle F(x),\varphi (x,\lambda _{n})\right\rangle =\int_{0}^{\pi }
\widetilde{\varphi }(t,\lambda _{n})f(t)\rho (t)dt-f_{3}-\frac{f_{4}}{\beta_{n}}=0
\] is found. From here and (\ref{19}), $Res_{\lambda =\lambda _{n}}y\left( x,\lambda \right) =0$ is obtained. 
Hence, $y(x,\lambda )$ is entire
function with respect to $\lambda $ for each fixed $x\in \left[ 0,\pi \right].$ Taking into account 
\begin{equation} \label{d}
\left|\Delta (\lambda)\right|\geq \left\vert 
\lambda \right\vert ^{2}C_{\delta }\exp(\left\vert Im\lambda \right\vert \mu (\pi))
\end{equation} 
and the following equalities being valid according to (\cite{Mar}, Lemma 1.3.1)
\begin{equation} \label{20}
\lim_{\left\vert \lambda \right\vert \rightarrow \infty }\max_{0\leq x\leq
\pi }\frac{1}{\left\vert \lambda \right\vert }\exp \left(-\left\vert Im\lambda \right\vert
\mu (x)\right)\left\vert \int_{0}^{x}\widetilde{\varphi }(t,\lambda
)f(t)\rho (t)dt\right\vert =0,  
\end{equation}
\begin{equation}  \label{21}
\lim_{\left\vert \lambda \right\vert \rightarrow \infty }\max_{0\leq x\leq
\pi }\frac{1}{\left\vert \lambda \right\vert }\exp \left(-\left\vert Im\lambda \right\vert
\left( \mu (\pi )-\mu (x)\right) \right)\left\vert \int_{x}^{\pi }
\widetilde{\psi }(t,\lambda )f(t)\rho (t)dt\right\vert =0,  
\end{equation}
we have 
\[
\lim_{\left\vert \lambda \right\vert \rightarrow \infty }\max_{0\leq x\leq
\pi }\left\vert y(x,\lambda )\right\vert =0.
\] Consequently, $y(x,\lambda )\equiv 0$. From (\ref{17}) and (\ref{18}), $F(x)=0$ a.e.
on $(0,\pi )$ is obtained.
\end{proof}
\begin{theorem} 
Let $F\left( x\right) \in D(L).$ Then the following expansion formula holds:
\begin{equation}  \label{22}
f(x)=\sum_{n=-\infty }^{\infty }a_{n}\varphi (x,\lambda _{n}),  
\end{equation}
\[
f_{3}=\sum_{n=-\infty }^{\infty }a_{n}\left[ b_{3}\varphi _{2}\left(
0,\lambda _{n}\right) +b_{4}\varphi _{1}\left( 0,\lambda _{n}\right) \right],
\]
\[
f_{4}=\sum_{n=-\infty }^{\infty }a_{n}\left[ b_{3}\varphi _{2}\left( \pi
,\lambda _{n}\right) +b_{4}\varphi _{1}\left( \pi ,\lambda _{n}\right) 
\right] ,
\] where
\[
a_{n}=\frac{1}{\alpha _{n}}\left\langle f(x),\varphi (x,\lambda
_{n})\right\rangle .
\] The series converges uniformly with respect to $x\in \left[ 0,\pi \right] $.
The series (\ref{22}) converges in $L_{2,\rho}(0,\pi;\mathbb{C}^{2})$ for $f(x)\in L_{2,\rho}(0,\pi;\mathbb{C}^{2})$ and Parseval equality
\begin{equation}  \label{23}
\left\Vert f\right\Vert ^{2}=\sum_{n=-\infty }^{\infty }\alpha
_{n}\left\vert a_{n}\right\vert ^{2}  
\end{equation}
is valid.
\end{theorem}
\begin{proof}
Since $\varphi (x,\lambda )$ and $\psi (x,\lambda )$ are solution of the
problem (\ref{1}),(\ref{2}),
\begin{eqnarray}
y(x,\lambda ) &=&-\frac{1}{\lambda \Delta (\lambda )}\psi (x,\lambda
)\int_{0}^{x}\left\{ -\frac{\partial }{\partial t}\widetilde{\varphi } \nonumber
(t,\lambda )B+\widetilde{\varphi }(t,\lambda )\Omega (t)\right\} f(t)dt \\ \nonumber
&&-\frac{1}{\lambda \Delta (\lambda )}\varphi (x,\lambda )\int_{x}^{\pi
}\left\{ -\frac{\partial }{\partial t}\widetilde{\psi }(t,\lambda )B+
\widetilde{\psi }(t,\lambda )\Omega (t)\right\} f(t)dt  \\ \nonumber
&&+\frac{f_{4}}{\Delta (\lambda )}\varphi (x,\lambda )+\frac{f_{3}}{\Delta
(\lambda )}\psi (x,\lambda )  \nonumber
\end{eqnarray}
can be written. Integrating by parts and using the expression of Wronskian
\begin{equation} \label{24}
y(x,\lambda )=-\frac{1}{\lambda }f(x)-\frac{1}{\lambda }z(x,\lambda )+\frac{
f_{4}}{\Delta (\lambda )}\varphi (x,\lambda )+\frac{f_{3}}{\Delta (\lambda )}
\psi (x,\lambda )  
\end{equation}
is obtained, where
\[
z(x,\lambda)=\frac{1}{\Delta (\lambda )}\left\{ \psi (x,\lambda
)\int_{0}^{x}\widetilde{\varphi }(t,\lambda )Bf^{^{\prime }}(t)dt+\varphi
(x,\lambda )\int_{x}^{\pi }\widetilde{\psi }(t,\lambda )Bf^{^{\prime
}}(t)dt+\right.
\]
\[
\ \ \ \ \ \ \ \ \ \ \ +\left. \psi (x,\lambda )\int_{0}^{x}\widetilde{\varphi}
(t,\lambda)\Omega (t)f(t)dt+\varphi (x,\lambda )\int_{x}^{\pi}\widetilde{\psi}
(t,\lambda)\Omega(t)f(t)dt\right\}.
\]
It follows from (\ref{20}) and (\ref{21}) that 
\begin{equation}  \label{25}
\lim_{\left\vert \lambda \right\vert \rightarrow \infty }\max_{0\leq x\leq
\pi }\left\vert z(x,\lambda )\right\vert =0,\ \ \ \ \lambda \in G_{\delta }.
\end{equation}
Now, we integrate $y(x,\lambda )$ with respect to $\lambda $ over the
contour $\Gamma _{N}$ with oriented counter clockwise as follows:
\[
I_{N}(x)=\frac{1}{2\pi i}\oint_{\Gamma _{N}}y(x,\lambda )d\lambda ,
\] where
\[
\Gamma _{N}=\left\{ \lambda :\left\vert \lambda \right\vert =\left( N+\frac{1
}{\pi }\arctan \left(\frac{c_{3}b_{4}-c_{4}b_{3}}{b_{3}c_{3}+c_{4}b_{4}}\right)\right) 
\frac{\pi }{\mu \left( \pi \right) }+\frac{\pi }{2\mu (\pi )}\right\} ,
\] $N$ is sufficiently large natural number. Applying Residue theorem, we have
\begin{eqnarray} \nonumber
\lefteqn{I_{N}(x)=\sum_{n=-N}^{N}\underset{\lambda =\lambda_{n}}{Res}\ y(x,\lambda )} \\ \nonumber
&=&-\sum_{n=-N}^{N}\frac{1}{\alpha _{n}}\varphi (x,\lambda
_{n})\int_{0}^{\pi }\widetilde{\varphi }(t,\lambda _{n})f(t)\rho(t)dt
+\sum_{n=-N}^{N}\frac{f_{4}}{\dot{\Delta}(\lambda _{n})}\varphi(x,\lambda _{n}) 
+\sum_{n=-N}^{N}\frac{f_{3}}{\dot{\Delta}(\lambda _{n})}\psi(x,\lambda _{n}). \nonumber
\end{eqnarray}
On the other hand, taking into account the equation (\ref{24})
\begin{equation}  \label{26}
f(x)=\sum_{n=-N}^{N}a_{n}\varphi (x,\lambda _{n})+\epsilon _{N}(x)
\end{equation}
is found, where 
\[
\epsilon _{N}(x)=-\frac{1}{2\pi i}\oint_{\Gamma _{N}}\frac{1}{\lambda}z(x,\lambda )d\lambda
\] and 
\[
a_{n}=\frac{1}{\alpha _{n}}\int_{0}^{\pi }\widetilde{\varphi }(t,\lambda_{n})f(t)\rho (t)dt.
\] From (\ref{25}), ${\lim }_{N\rightarrow
\infty}{\max}_{0\leq x\leq \pi }\left\vert \epsilon
_{N}(x)\right\vert =0.$ Thus, by going over in (\ref{26}) to the limit as $
N\rightarrow \infty $ the expansion formula (\ref{22}) with respect to
eigenfunction is obtained. Since the system of $\varphi (x,\lambda
_{n}),(n\in \mathbb{Z})$ is complete and orthogonal in $L_{2,\rho}(0,\pi;\mathbb{C}^{2})\oplus \mathbb{C}^{2}$, 
Parseval equality (\ref{23}) is valid. 
\end{proof}

\section{Uniqueness Theorem by Weyl Function}
In this section, we define Weyl function and Weyl solution. Uniqueness theorem for 
inverse problem according to Weyl function is proved.

Denote by $\Phi(x,\lambda)=\left(
\begin{array}{c}
\Phi_{1}(x,\lambda) \\
\Phi_{2}(x,\lambda)
\end{array}
\right)$ the solution of the system (\ref{1}),
satisfying the conditions 
\begin{eqnarray} \nonumber
b_{1}\Phi _{2}(0,\lambda )+b_{2}\Phi _{1}(0,\lambda )-\lambda \left(
b_{3}\Phi _{2}(0,\lambda )+b_{4}\Phi _{1}(0,\lambda )\right) &=&1, \\ \nonumber
\\
c_{1}\Phi _{2}(\pi ,\lambda )+c_{2}\Phi _{1}(\pi ,\lambda )+\lambda \left(
c_{3}\Phi _{2}(\pi ,\lambda )+c_{4}\Phi _{1}(\pi ,\lambda )\right) &=& 0. \nonumber
\end{eqnarray}
The function $\Phi (x,\lambda )$ is called Weyl solution of the problem (\ref{1}),(\ref{2}). 
Let the function $C(x,\lambda)=\left(
\begin{array}{c}
C_{1}(x,\lambda) \\
C_{2}(x,\lambda)
\end{array}
\right)$ be the solution of system (\ref{1}), satisfying the initial condition 
\[
C_{1}(0,\lambda )=-\frac{b_{3}}{k_{1}},\ \ \ \ C_{2}(0,\lambda )=\frac{b_{4}}{k_{1}}.
\] As in Lemma \ref{lma2}, it is obtained that $C(x,\lambda)=\left(
\begin{array}{c}
C_{1}(x,\lambda) \\
C_{2}(x,\lambda)
\end{array}
\right)$ has the following integral representation
\begin{eqnarray} \nonumber
C_{1}(x,\lambda)&=&-\frac{b_{3}}{k_{1}}cos\lambda\mu(x)
-\frac{b_{4}}{k_{1}}sin\lambda\mu(x)
-\frac{b_{3}}{k_{1}}\int_{0}^{\mu(x)}\left[\tilde{B}_{11}(x,t)
cos\lambda t+\stackrel{\approx}{B}_{12}(x,t)sin\lambda t\right]dt- \\ \nonumber
&&-\frac{b_{4}}{k_{1}}\int_{0}^{\mu(x)}\left[\stackrel{\approx}{B}_{11}(x,t)sin\lambda t-\tilde{B}_{12}(x,t)cos\lambda t\right]dt, \nonumber
\end{eqnarray}
\begin{eqnarray} \nonumber
C_{2}(x,\lambda)&=&\frac{b_{4}}{k_{1}}cos\lambda\mu(x)
-\frac{b_{3}}{k_{1}}sin\lambda\mu(x) 
-\frac{b_{3}}{k_{1}}\int_{0}^{\mu(x)}\left[\tilde{B}_{21}(x,t)
cos\lambda t+\stackrel{\approx}{B}_{22}(x,t)sin\lambda t\right]dt- \\ \nonumber
&&-\frac{b_{4}}{k_{1}}\int_{0}^{\mu(x)}\left[\stackrel{\approx}{B}_{21}(x,t)sin\lambda t
-\tilde{B}_{22}(x,t)cos\lambda t\right]dt, \nonumber
\end{eqnarray}
where $\tilde{B}_{ij}\left( x,.\right)\in L_{2}\left( 0,\pi \right),$ $\stackrel{\approx }{B}_{ij}\left( x,.\right)\in L_{2}\left( 0,\pi\right) ,$ $i,j=1,2.$
The solution $\psi (x,\lambda)$ can be shown that 
\begin{equation} \label{27}
\frac{\psi (x,\lambda )}{\Delta \left( \lambda \right) }=C\left( x,\lambda
\right) -\frac{\left( b_{4}\psi _{1}\left( 0,\lambda \right) +b_{3}\psi
_{2}\left( 0,\lambda \right) \right) }{k_{1}\Delta \left( \lambda \right)}
\varphi \left( x,\lambda \right) .  
\end{equation}
Denote 
\begin{equation}  \label{28}
M(\lambda ):=-\frac{\left( b_{4}\psi _{1}\left( 0,\lambda \right) +b_{3}\psi
_{2}\left( 0,\lambda \right) \right) }{k_{1}\Delta \left( \lambda \right) }.
\end{equation}
It is obvious that 
\begin{equation}  \label{29}
\Phi (x,\lambda )=C(x,\lambda )+M(\lambda )\varphi (x,\lambda).  
\end{equation}
The function 
\[M(\lambda)=-\frac{\left( b_{4}\Phi _{1}\left( 0,\lambda
\right) +b_{3}\Phi _{2}\left( 0,\lambda \right) \right) }{k_{1}}\]
is called the Weyl function of the boundary value problem (\ref{1}),(\ref{2}). The
Weyl solution and Weyl function are meromorphic functions having simple
poles at points $\lambda _{n}$ eigenvalues of problem (\ref{1}),(\ref{2}).
It is obtained from (\ref{27}) and (\ref{29}) that 
\begin{equation}  \label{30}
\Phi (x,\lambda )=\frac{\psi (x,\lambda )}{\Delta (\lambda )}.  
\end{equation}
\begin{theorem} 
For the Weyl function $M(\lambda )$, the following representation holds: 
\begin{equation}  \label{33}
M(\lambda)=\sum_{n=-\infty }^{\infty}\frac{1}{\alpha _{n}(\lambda -\lambda _{n})}.  
\end{equation}
\end{theorem}
\begin{proof}
Since 
\begin{eqnarray}\nonumber
W[C(x,\lambda),\psi(x,\lambda)]&=&C_{2}(x,\lambda)\psi_{1}(x,\lambda)-C_{1}(x,\lambda)\psi_{2}(x,\lambda) \\ \nonumber
\\ \nonumber
&=&\frac{b_{4}\psi_{1}(0,\lambda)+b_{3}\psi(0,\lambda)}{k_{1}} \\ \nonumber
\\ \nonumber
&=&-\left[\left(c_{1}+\lambda c_{3}\right)C_{2}(\pi,\lambda)+\left(c_{2}+\lambda c_{4}\right)C_{1}(\pi,\lambda)\right], \nonumber
\end{eqnarray} we can rewrite the Weyl function (\ref{28}) as follows
\[
M(\lambda)=\frac{\left(c_{1}+\lambda c_{3}\right)C_{2}(\pi,\lambda)+\left(c_{2}+\lambda c_{4}\right)C_{1}(\pi,\lambda)}{\Delta(\lambda)}.
\]
Using the expression of solution $C(x,\lambda)$ and (\ref{d}), we have
\begin{equation}\label{35}
\lim_{\stackrel{\left\vert \lambda \right\vert \rightarrow \infty }{\lambda
\in G_{\delta }}}\left|M(\lambda)\right|=0.
\end{equation}
Since $\psi(x,\lambda_{n})=\beta_{n}\varphi(x,\lambda_{n})$,
\[
\beta_{n}=-\frac{b_{4}\psi
_{1}\left( 0,\lambda _{n}\right) +b_{3}\psi _{2}(0,\lambda _{n})}{k_{1}}.
\] Then, we get
\begin{equation}\label{c}
\underset{\lambda =\lambda _{n}}{Res}\ M(\lambda) =-\frac{b_{4}\psi
_{1}\left( 0,\lambda _{n}\right) +b_{3}\psi _{2}(0,\lambda _{n})}{k_{1}\dot{
\Delta}(\lambda _{n})}=\frac{1}{\alpha _{n}}.
\end{equation} 
Consider the following contour integral 
\[
I_{N}(\lambda )=\frac{1}{2\pi i}\int_{\Gamma _{N}(\lambda )}\frac{M(\xi)}
{\xi -\lambda }d\xi ,\ \ \ \xi \in int\Gamma _{N},
\] where 
\[
\Gamma _{N}=\left\{\lambda :\left\vert \lambda \right\vert =\left( N+\frac{1
}{\pi }\arctan \left(\frac{c_{3}b_{4}-c_{4}b_{3}}{b_{3}c_{3}+c_{4}b_{4}}\right)\right) 
\frac{\pi }{\mu \left( \pi \right) }+\frac{\pi }{2\mu (\pi )}\right\} .
\] It follows from (\ref{35}) that ${\lim}_{N\rightarrow \infty}
I_{N}(\lambda )=0.$ On the other hand, applying Residue theorem and the
residue (\ref{c}),
\[
I_{N}(\lambda )=M(\lambda )+\sum_{\lambda _{n}\in int\Gamma
_{N}}\frac{1}{\alpha _{n}(\lambda _{n}-\lambda )}
\] is found.
Thus, as $N\rightarrow \infty$ 
\[
M(\lambda)=\sum_{n=-\infty}^{\infty}\frac{1}{\alpha_{n}(\lambda -\lambda_{n})}
\]
is obtained.
\end{proof}

Now, we seek inverse problem of the reconstruction of the problem (\ref{1}), (\ref{2}) 
by Weyl function $M\left( \lambda \right) $ and spectral data $
\left\{ \lambda _{n},\alpha _{n}\right\} ,(n\in \mathbb{Z})$. Along with problem (\ref{1}),(\ref{2}), we consider a boundary value problem of the same
form, but with another potential function $\tilde{\Omega}(x)$. Let's agree to that if some symbol s denotes
an object relating to the problem (\ref{1}),(\ref{2}), then $\tilde{s}$ will denote an object, relating to
the boundary value problem with the function $\tilde{\Omega}(x)$.
\begin{theorem} \label{the} 
If $M(\lambda )=\tilde{M}(\lambda )$, $\Omega(x)=\tilde{\Omega}(x)$, i.e. the boundary value problem (\ref{1}), (\ref{2}) is uniquely determined by the Weyl function.
\end{theorem}
\begin{proof}
We describe the matrix $P(x,\lambda )=\left[ P_{ij}(x,\lambda )\right]_{i,j=1,2}$ with the formula 
\begin{equation}  \label{38}
P(x,\lambda )\left( 
\begin{array}{cc}
\tilde{\varphi _{1}} & \tilde{\Phi}_{1} \\ 
\tilde{\varphi _{2}} & \tilde{\Phi}_{2}
\end{array}
\right) =\left( 
\begin{array}{cc}
\varphi _{1} & \Phi _{1} \\ 
\varphi _{2} & \Phi _{2}
\end{array}
\right).  
\end{equation}
The Wronskian of the solutions $\tilde{\varphi}(x,\lambda )$ and $\ \tilde{\Phi}(x,\lambda )$ is 
\begin{equation}  \label{39}
W[\tilde{\varphi}(x,\lambda ),\tilde{\Phi}(x,\lambda )]=\tilde{\varphi _{2}}
(x,\lambda )\tilde{\Phi _{1}}(x,\lambda )-\tilde{\varphi _{1}}(x,\lambda )
\tilde{\Phi _{2}}(x,\lambda )=1.  
\end{equation}
Using (\ref{38}) and (\ref{39}), we calculate 
\begin{equation}  \label{40} 
\begin{array}{cc}
P_{11}(x,\lambda )=\Phi _{1}(x,\lambda )\tilde{\varphi}_{2}(x,\lambda
)-\varphi _{1}(x,\lambda )\tilde{\Phi}_{2}(x,\lambda ), \\
P_{12}(x,\lambda )=\varphi _{1}(x,\lambda )\tilde{\Phi}_{1}(x,\lambda
)-\Phi _{1}(x,\lambda )\tilde{\varphi}_{1}(x,\lambda ), \\
P_{21}(x,\lambda )=\Phi _{2}(x,\lambda )\tilde{\varphi}_{2}(x,\lambda
)-\varphi _{2}(x,\lambda )\tilde{\Phi}_{2}(x,\lambda ), \\
P_{22}(x,\lambda )=\varphi _{2}(x,\lambda )\tilde{\Phi}_{1}(x,\lambda
)-\Phi _{2}(x,\lambda )\tilde{\varphi}_{1}(x,\lambda )
\end{array}
\end{equation}
and
\begin{equation}   \label{41}
\begin{array}{cc}
\varphi _{1}(x,\lambda )=P_{11}(x,\lambda )\tilde{\varphi}_{1}(x,\lambda
)+P_{12}(x,\lambda )\tilde{\varphi}_{2}(x,\lambda ),  \\
\varphi _{2}(x,\lambda )=P_{21}(x,\lambda )\tilde{\varphi}_{1}(x,\lambda
)+P_{22}(x,\lambda )\tilde{\varphi}_{2}(x,\lambda ),  \\
\Phi _{1}(x,\lambda )=P_{11}(x,\lambda )\tilde{\Phi}_{1}(x,\lambda
)+P_{12}(x,\lambda )\tilde{\Phi}_{2}(x,\lambda ),  \\
\Phi _{2}(x,\lambda )=P_{21}(x,\lambda )\tilde{\Phi}_{1}(x,\lambda
)+P_{22}(x,\lambda )\tilde{\Phi}_{2}(x,\lambda ). 
\end{array}
\end{equation}
Taking into account (\ref{30}), (\ref{39}) and (\ref{40}), 
\[
P_{11}(x,\lambda)-1=\frac{\tilde{\psi}_{2}(x,\lambda)}{\tilde{\Delta}
(\lambda)}\{\tilde{\varphi}_{1}(x,\lambda)-\varphi_{1}(x,\lambda)\} 
-\tilde{\varphi}_{2}(x,\lambda)\left\{\frac{\tilde{\psi}_{1}(x,\lambda)}
{\tilde{\Delta}(\lambda)}-\frac{\psi_{1}(x,\lambda)}{\Delta(\lambda)}\right\} 
\]
\[
P_{12}(x,\lambda)=\frac{\psi_{1}(x,\lambda)}{\Delta(\lambda)}\{\varphi_{1}(x,\lambda)
-\tilde{\varphi}_{1}(x,\lambda)\}+\varphi_{1}(x,\lambda)
\left\{\frac{\tilde{\psi}_{1}(x,\lambda )}{\tilde{\Delta}
(\lambda)}-\frac{\psi _{1}(x,\lambda )}{\Delta(\lambda)}\right\}
\]
\[
P_{21}(x,\lambda)=\frac{\psi_{2}(x,\lambda )}{\Delta(\lambda)}
\{\tilde{\varphi}_{2}(x,\lambda)-\varphi_{2}(x,\lambda)\} 
+\varphi_{2}(x,\lambda)\left\{\frac{\psi _{2}(x,\lambda)}{\Delta(\lambda)}-
\frac{\tilde{\psi}_{2}(x,\lambda )}{\tilde{\Delta}(\lambda )}\right\}
\]
\[
P_{22}(x,\lambda)-1=\tilde{\varphi}_{1}(x,\lambda)
\left\{\frac{\tilde{\psi}_{2}(x,\lambda)}
{\tilde{\Delta}(\lambda)}-\frac{\psi_{2}(x,\lambda)}{\Delta(\lambda)}\right\} 
-\frac{\tilde{\psi}_{1}(x,\lambda)}{\tilde{\Delta}(\lambda)}
\{\tilde{\varphi}_{2}(x,\lambda)-\varphi_{2}(x,\lambda)\}
\] 
are found. Using (\ref{d}), we obtain 
\begin{equation} \label{42}
\begin{array}{cc}
\lim_{\stackrel{\left\vert \lambda \right\vert \rightarrow \infty }{\lambda
\in G_{\delta }}}\max_{0\leq x\leq \pi }\left\vert P_{11}(x,\lambda)-1\right\vert  =0,\\
\lim_{\stackrel{\left\vert \lambda \right\vert
\rightarrow \infty }{\lambda \in G_{\delta }}}\max_{0\leq x\leq \pi
}\left\vert P_{22}(x,\lambda )-1\right\vert =0,  \\
\lim_{\stackrel{\left\vert \lambda \right\vert \rightarrow \infty }{\lambda
\in G_{\delta }}}\max_{0\leq x\leq \pi }\left\vert P_{12}(x,\lambda
)\right\vert =0,\\
\lim_{\stackrel{\left\vert \lambda \right\vert
\rightarrow \infty }{\lambda \in G_{\delta }}}\max_{0\leq x\leq \pi
}\left\vert P_{21}(x,\lambda )\right\vert =0. 
\end{array}
\end{equation}
Substituting (\ref{29}) into (\ref{40}), we have 
\[
P_{11}(x,\lambda)=C_{1}(x,\lambda)\tilde{\varphi}_{2}(x,\lambda )-\varphi
_{1}(x,\lambda )\tilde{C}_{2}(x,\lambda)
+\varphi _{1}(x,\lambda )\tilde{\varphi}_{2}(x,\lambda)
\left[ M(\lambda)-\tilde{M}(\lambda)\right], 
\]
\[
P_{12}(x,\lambda)=\varphi_{1}(x,\lambda)\tilde{C}_{1}(x,\lambda)-C_{1}(x,\lambda )
\tilde{\varphi}_{1}(x,\lambda)
+\varphi_{1}(x,\lambda)\tilde{\varphi}_{1}(x,\lambda)\left[\tilde{M}(\lambda)-M(\lambda)\right],  
\]
\[
P_{21}(x,\lambda)=C_{2}(x,\lambda)\tilde{\varphi}_{2}(x,\lambda)
-\varphi_{2}(x,\lambda)\tilde{C}_{2}(x,\lambda)
+\varphi_{2}(x,\lambda)\tilde{\varphi}_{2}(x,\lambda)\left[M(\lambda)-\tilde{M}(\lambda)\right], 
\]
\[
P_{22}(x,\lambda)=\varphi_{2}(x,\lambda)\tilde{C}_{1}(x,\lambda)
-C_{2}(x,\lambda)\tilde{\varphi}_{1}(x,\lambda) 
+\varphi_{2}(x,\lambda)\tilde{\varphi}_{1}(x,\lambda)\left[\tilde{M}(\lambda)-M(\lambda)\right]. 
\]
Hence, if $M(\lambda)\equiv \tilde{M}(\lambda )$, ${P_{ij}(x,\lambda)_{i,j=1,2}}$ 
are entire functions with respect to $\lambda $ for every
fixed $x$. Then from (\ref{42}), we find 
\[
P_{11}(x,\lambda )\equiv 1,\ \ \ P_{12}(x,\lambda )\equiv 0,
\]
\[
P_{21}(x,\lambda )\equiv 0,\ \ \ P_{22}(x,\lambda )\equiv 1.
\]
Substituting these identities into (\ref{41}), 
\[
\varphi _{1}(x,\lambda ) \equiv \tilde{\varphi}_{1}(x,\lambda ),\ \ \
\varphi _{2}(x,\lambda ) \equiv \tilde{\varphi}_{2}(x,\lambda ),
\]
\[
\Phi _{1}(x,\lambda ) \equiv \tilde{\Phi}_{1}(x,\lambda ),\ \ \ \Phi
_{2}(x,\lambda )\equiv \tilde{\Phi}_{2}(x,\lambda )
\] are obtained for all $x$ and $\lambda $, so $\Omega(x)\equiv \tilde{\Omega}(x).$
\end{proof}

According to (\ref{33}), the specification of the Weyl function $M(\lambda)$
is equivalent to the specification of the spectral data $\left\{\lambda_{n},\ \alpha_{n}\right\}$, $n\in \mathbb{Z}.$ That is,
if $\lambda _{n}=\tilde{\lambda}_{n},$ $\alpha _{n}=\tilde{\alpha}_{n}$ for all $n\in \mathbb{Z}$, 
$M(\lambda)=\tilde{M}(\lambda )$ is obtained.
It follows from Theorem \ref{the} that $\Omega(x)=\tilde{\Omega}(x)$. We have thus proved the following theorem:
\begin{theorem}
The problem (\ref{1}),(\ref{2}) is uniquely determined by spectral data 
$\left\{\lambda_{n},\ \alpha_{n}\right\}$, $n\in \mathbb{Z}.$
\end{theorem}

\end{document}